\documentclass[12pt,a4paper]{amsart}

\usepackage{hyperref} 
\usepackage{amssymb,amsrefs}
\usepackage[all]{xy}

\newcommand{\bbA}{\mathbb{A}}
\newcommand{\bbN}{\mathbb{N}}
\newcommand{\bbZ}{\mathbb{Z}}

\newcommand{\calA}{\mathcal{A}}
\newcommand{\calC}{\mathcal{C}}
\newcommand{\calM}{\mathcal{M}}
\newcommand{\calN}{\mathcal{N}}
\newcommand{\calO}{\mathcal{O}}
\newcommand{\calR}{\mathcal{R}}

\newcommand{\calV}{\mathcal{V}}

\newcommand{\bone}{\mathbf{1}}

\newcommand{\bDelta}{\mathbf{\Delta}}
\newcommand{\bGamma}{\mathbf{\Gamma}}

\let\mod=\undefined
\DeclareMathOperator{\End}{End} %
\DeclareMathOperator{\Ext}{Ext} %
\DeclareMathOperator{\Hom}{Hom} %
\DeclareMathOperator{\mod}{mod} %
\DeclareMathOperator{\Gdim}{Gdim} %

\newcounter{claim}[section]

\newtheorem{corollary}[claim]{Corollary}
\newtheorem{lemma}[claim]{Lemma}
\newtheorem{proposition}[claim]{Proposition}
\newtheorem{theorem}[claim]{Theorem}

\newcounter{result}

\newtheorem{maintheorem}[result]{Theorem}
\newtheorem{maincorollary}[result]{Corollary}

\newcommand{\vertexD}[1]{\bullet \save*+!D{\scriptstyle #1} \restore}

\newcommand{\vertexU}[1]{\bullet \save*+!U{\scriptstyle #1} \restore}

\title[Gentle cluster tilted algebras]%
  {The algebras derived equivalent to gentle cluster tilted algebras}

\author{Grzegorz Bobi\'nski}

\address{Faculty of Mathematics and Computer Science \\ Nicolaus
Copernicus University \\ ul.~Chopina 12/18 \\ 87-100 Toru\'n \\
Poland}

\email{gregbob@mat.uni.torun.pl}

\author{Aslak Bakke Buan}

\address{Department of Mathematical Sciences \\ Norwegian
University of Science and Technology \\ 7491 Trondheim \\ Norway}

\email{aslakb@math.ntnu.no}

\thanks{The authors were supported by STOR-FORSK grant 196600 from
NFR. The first named author also acknowledges the support from the
Research Grant No.\ N N201 269135 of the Polish Ministry of
Science and Higher Education and the Humboldt Foundation.}

\date{}



\begin{document}

\begin{abstract}
A cluster tilted algebra is known to be gentle if and only if it
is cluster tilted of Dynkin type $\bbA$ or Euclidean type
$\tilde{\bbA}$. We classify all finite dimensional algebras which
are derived equivalent to gentle cluster tilted algebras.
\end{abstract}

\maketitle

We consider finite dimensional algebras over an algebraically
closed field $k$. Dealing with such algebras up to Morita
equivalence, we may assume that they are given as path algebras
modulo ideals of relations. Gentle algebras form a particularly
nice subclass of special biserial algebras. This is a  well
understood and much studied class of algebras of tame
representation type. They occur in various settings related to
group algebras of finite groups, and also frequently as test
classes when dealing with general problems for finite dimensional
algebras.

The special biserial algebras can be combinatorially characterized
in terms of their quivers and relations, and so can the subclass
of gentle algebras, which we will study here. A prominent class of
examples comes from path algebras; a hereditary algebra is easily
seen to be gentle if and only if it is the path algebra of a
quiver of Dynkin type $\bbA$ or Euclidean type $\tilde{\bbA}$.

When dealing with questions of a homological nature, one is
frequently inclined to study algebras up to derived equivalence.
Two algebras are said to be derived equivalent if their derived
categories are equivalent as triangulated categories, and this
happens if and only if one can get from one algebra to the other
by taking the endomorphism ring of a so called tilting complex. A
special case of tilting complexes are tilting modules.

By~\cite{SZ}, the class of gentle algebras is closed under derived
equivalence. Hence the characterization of hereditary gentle
algebras implies also a characterization of gentle algebras
derived equivalent to hereditary algebras.

Recently, a class of algebras with representation theory very
similar to that of hereditary algebras has been much studied;
these are the cluster tilted algebras. Such an algebra is defined
to be the endomorphism ring $\End_{\calC_{k Q}} (T)$ of a tilting
object $T$ in the cluster category $\calC_{k Q}$ of a quiver $Q$
and is then said to be cluster tilted of type $Q$.

Cluster tilted algebras of type $\bbA$ were classified
in~\cite{BV}, while cluster tilted algebras of type $\tilde{\bbA}$
were classified in~\cite{Ba}. They are in both cases gentle, and
moreover, Assem et.\ al~\cite{ABChJP} have shown that no other
cluster tilted algebras are gentle. Furthermore, in both cases,
also a classification of the derived equivalence classes were
given.

The main aim of this paper is to give a complete classification of
all finite dimensional algebras which are derived equivalent to
gentle cluster tilted algebras. For this we use work of
Avella-Alaminos and Geiss~\cite{AAG}. They showed that one can
assign to each gentle algebra $\Lambda$ a function $f_{\Lambda}
\colon \bbN^2 \to \bbN$, and that this function is invariant under
derived equivalence. The function can be algorithmically computed
from the quiver and relations of the algebra. Our classification
is described in terms of this function. More precisely, we give
necessary and sufficient conditions on $f_\Lambda$ for $\Lambda$
to be derived equivalent to a cluster tilted algebra of type
$\bbA$, and similarly we give conditions for type $\tilde{\bbA}$.

We also point out that in the case of algebras which are derived
equivalent to gentle cluster tilted algebras, we have that
$f_\Lambda$ uniquely determines the derived equivalence class of
$\Lambda$.

Another main tool is a combinatorial description of
Brenner--Butler (co)tilting~\cite{BB} (shortly, BB-(co)tilting).
In fact, as a consequence of the proof of our main result, it
follows that any derived equivalence between two gentle algebras
derived equivalent to cluster tilted algebras can be obtained by
repeated BB-tilting or BB-cotilting.

It is known that the cluster algebras are Gorenstein algebras of
Gorenstein dimension $1$. We show that this property characterizes
the gentle cluster algebras among the algebras derived equivalent
to gentle cluster algebras.

The paper is organized as follows. In Section~\ref{section_gentle}
we collect facts about gentle algebras, while in
Section~\ref{section_tilting} we present basics about derived
equivalence and define Brenner--Butler (co)tilting modules. Next,
in Section~\ref{section_AAG} we define the invariant of
Avella--Alaminos and Geiss and in Section~\ref{section_completion}
we introduce a combinatorial construction used in the proofs.
Section~\ref{section_cluster} is devoted to a presentation of
known facts about gentle cluster tilted algebras. We also
formulate the main results of the paper there. Finally, the last
four sections contain the proofs of the main results.

We refer to~\cites{ARS, ASS} for general notions, and to~\cite{H}
for derived categories.

\section{Gentle algebras} \label{section_gentle}

Throughout the paper $k$ is a fixed algebraically closed field. By
$\bbZ$, $\bbN$ and $\bbN_+$ we denote the sets of the integers,
the non-negative integers and the positive integers, respectively.
Finally, if $i, j \in \bbZ$, then $[i, j] := \{ k \in \bbZ \mid i
\leq k \leq j \}$ (in particular, $[i, j] = \varnothing$ if $i >
j$).

By a quiver $\Delta$ we mean a (non-empty) finite set $\Delta_0$
of vertices and a finite set $\Delta_1$ of arrows together with
two maps $s = s_\Delta, t = t_\Delta : \Delta_1 \to \Delta_0$
which assign to $\alpha \in \Delta_1$ the starting vertex $s
\alpha$ and the terminating vertex $t \alpha$, respectively. A
vertex $x$ of a quiver $\Delta$ is said to be adjacent to an arrow
$\alpha$ if $x \in \{ s \alpha, t \alpha \}$. Similarly, arrows
$\alpha$ and $\beta$ are said to be adjacent if $\{ s \alpha, t
\alpha \} \cap \{ s \beta, t \beta \} \neq \varnothing$. A quiver
$\Delta$ is called connected if for all $x, y \in \Delta_0$, $x
\neq y$, there exists a sequence $(\alpha_1, \ldots, \alpha_n)$ of
arrows such that $x$ is adjacent to $\alpha_1$, $\alpha_i$ is
adjacent to $\alpha_{i + 1}$ for each $i \in [1, n - 1]$, and $y$
is adjacent to $\alpha_n$. If $\Delta$ is a quiver and $\Delta_1'
\subset \Delta_1$, then by the subquiver of $\Delta$ generated by
$\Delta_1'$ we mean the quiver $(\{ s \alpha, t \alpha \mid \alpha
\in \Delta_1' \}, \Delta_1' )$.

Fix a quiver $\Delta$. If $n \in \bbN_+$, then by a path in
$\Delta$ of length $n$ we mean a sequence $\omega = (\alpha_1,
\ldots, \alpha_n)$ such that $\alpha_i \in \Delta_1$ for each $i
\in [1, n]$ and $s \alpha_i = t \alpha_{i + 1}$ for each $i \in
[1, n - 1]$. In the above situation we put $s \omega := s
\alpha_n$ and $t \omega := t \alpha_1$. Moreover, for each $x \in
\Delta_0$ we introduce the trivial path $\bone_x$ at $x$ of length
$0$ such that $s \bone_x := x =: t \bone_x$. For a path $\omega$
we denote by $\ell (\omega)$ its length. If $\omega'$ and
$\omega''$ are paths in $\Delta$ of lengths $n'$ and $n''$,
respectively, such that $s \omega' = t \omega''$, then we define
the composition $\omega' \cdot \omega''$ of $\omega'$ and
$\omega''$, which is a path in $\Delta$ of length $n' + n''$, in
the obvious way (in particular, $\omega \cdot \bone_{s \omega} =
\omega = \bone_{t \omega} \cdot \omega$ for each path $\omega$).
We say that a path $\omega_0$ is a subpath of a path $\omega$ if
there exist paths $\omega'$ and $\omega''$ such that $\omega =
\omega' \cdot \omega_0 \cdot \omega''$.

By a quiver with (monomial) relations we mean a pair $\bDelta =
(\Delta, R)$ consisting of a quiver $\Delta$ and a set $R$ of
paths in $\Delta$. Given a quiver with relations $\bDelta$ we
define the algebra $k \bDelta$ in the following way. As a vector
space $k \bDelta$ has a basis formed by the paths in $\Delta$
which do not have a subpath from $R$. If $\omega'$ and $\omega''$
are two such paths, then their product is either $\omega' \cdot
\omega''$ provided $s \omega' = t \omega''$ and $\omega' \cdot
\omega''$ does not have a subpath from $R$, or $0$ elsewhere. If
$R = \varnothing$, then one writes $k \Delta$ instead of $k
\bDelta$ and we call $k \Delta$ the path algebra of $\Delta$. By
abuse of terminology we will also call $k \bDelta$ the path
algebra of $\bDelta$.

By a gentle quiver we mean a quiver with relations $\bDelta$ such
that $\Delta$ is connected, $R$ consists of paths of length $2$,
and the following conditions are satisfied:
\begin{enumerate}

\item
for each vertex $x$ there are at most two arrows $\alpha$ such
that $s \alpha = x$ and at most two arrows $\beta$ such that $t
\beta = x$,

\item
for each arrow $\alpha$ there is at most one arrow $\beta$ such
that $s \beta = t \alpha$ and $(\beta, \alpha) \not \in R$ and at
most one arrow $\gamma$ such that $t \gamma = s \alpha$ and
$(\alpha, \gamma) \not \in R$,

\item
for each arrow $\alpha$ there is at most one arrow $\beta$ such
that $(\beta, \alpha) \in R$ (in particular, $s \beta = t \alpha$)
and at most one arrow $\gamma$ such that $(\alpha, \gamma) \in R$
(in particular, $t \gamma = s \alpha$),

\item
there exists $n \in \bbN$ such that every path $\omega$ in
$\Delta$ of length $n$ has a subpath from $R$ (i.e., $\dim_k k
\bDelta < \infty$).

\end{enumerate}
In other words, conditions~(1)--(3) mean that the most complicated
situation which can appear in the neighborhood of a given vertex
$x$ is the following
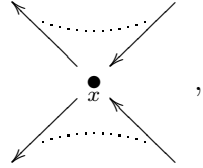
\[
\vcenter{\xymatrix{& & \ar[ld]_(0.25){}="b" \\ & \vertexU{x}
\ar[lu]_(0.75){}="a" \ar[ld]^(0.75){}="c" \\ & &
\ar[lu]^(0.25){}="d" \ar@{.}@/^/"b";"a" \ar@{.}@/_/"d";"c"}},
\]
where the dotted lines denote relations. An algebra $\Lambda$ is
called gentle if and only if there exists a gentle quiver
$\bDelta$ such that $\Lambda$ is isomorphic to $k \bDelta$.

\section{Derived equivalences and Brenner--Butler tilting modules}
\label{section_tilting}

For a finite dimensional algebra $\Lambda$ denote by $D^b
(\Lambda)$ the bounded derived category of the category of finite
dimensional right $\Lambda$-modules. Then $D^b (\Lambda)$ has a
structure of a triangulated category with the suspension functor
$\Sigma$ given by the shift of complexes. We say that finite
dimensional algebras $\Lambda$ and $\Lambda'$ are derived
equivalent if $D^b (\Lambda)$ and $D^b (\Lambda')$ are derived
equivalent as triangulated categories. Rickard~\cite{Ric} has
showed that this happens if and only if there exists a tilting
complex $T$ in $D^b (\Lambda)$ such that $\Lambda'$ is isomorphic
to $\End_{D^b (\Lambda)} (T)$. Recall that if $\Lambda$ is a
finite dimensional algebra, then a complex $T$ in $D^b (\Lambda)$
is called tilting if $\Hom_{D^b (\Lambda)} (T, \Sigma^i T) = 0$
for all $i \in \bbZ$, $i \neq 0$, and $T$ generates (as a
triangulated category) the full subcategory of $D^b (\Lambda)$
formed by the perfect complexes, where a complex is called perfect
if it is quasi-isomorphic to a bounded complex of projective
modules. A module is called tilting if it is a tilting complex,
when viewed as a complex concentrated in degree $0$. In the paper,
we consider a special class of tilting modules, co called
Brenner-Butler (co)tilting modules~\cite{BB}*{Chapter~2}. We
describe them for gentle algebras, but their definition
generalizes to arbitrary finite dimensional algebras.

Let $\bDelta$ be a gentle quiver without loops (i.e., there are no
arrows $\alpha$ in $\Delta$ such that $s \alpha = t \alpha$) and
$\Lambda := k \bDelta$. Let $x$ be a vertex in $\Delta$ such that
for each $\alpha \in \Delta_1$ with $s \alpha = x$ there exists
(necessarily unique) $\beta_\alpha \in \Delta_1$ with $t
\beta_\alpha = x$ and $(\alpha, \beta_\alpha) \not \in R$. Observe
that this condition is satisfied if there are no arrows starting
at $x$ or there are two arrows terminating at $x$. We define a
quiver with relations $\bDelta' = (\Delta', R')$, which we call
the quiver with relations obtained from $\bDelta$ by applying the
reflection at $x$, in the following way: $\Delta_0' = \Delta_0$,
$\Delta_1' = \Delta_1$,
\begin{align*}
s_{\Delta'} \alpha & :=
\begin{cases}
x & t_\Delta \alpha = x,
\\
s_\Delta \beta_\alpha & s_\Delta \alpha = x,
\\
s_\Delta \alpha & \text{otherwise},
\end{cases}
\\
t_{\Delta'} \alpha & :=
\begin{cases}
s_\Delta \alpha & t_\Delta \alpha = x,
\\
x & \exists \; \beta \in \Delta_1 : t_\Delta \beta = x \wedge
s_\Delta \beta = t_\Delta \alpha \wedge (\beta, \alpha) \in R,
\\
t_\Delta \alpha & \text{otherwise},
\end{cases}
\end{align*}
and
\begin{multline*}
R' := \{ (\alpha, \beta) \in R \mid t_\Delta \alpha \neq x \wedge
s_\Delta \alpha \neq x \} \cup \{ (\alpha, \beta_\alpha) \mid
s_\Delta \alpha = x\} \cup
\\
\{ (\alpha, \beta) \mid t_\Delta \alpha = x \wedge  \exists \;
\gamma \in \Delta_1 :
\\
\gamma \neq \alpha \wedge t_\Delta \gamma = x \wedge s_\Delta
\gamma = t_\Delta \beta \wedge (\gamma, \beta) \in R \}.
\end{multline*}
For example, if $\Delta$ is the following quiver
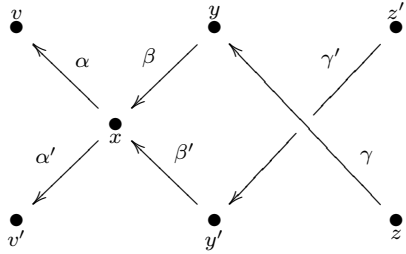
\[
\xymatrix{\vertexD{v} & & \vertexD{y} \ar[ld]_\beta & &
\vertexD{z'} \ar[lldd]_(.25){\gamma'}|\hole \\ & \vertexU{x}
\ar[lu]_\alpha \ar[ld]_{\alpha'} \\ \vertexU{v'} & & \vertexU{y'}
\ar[lu]_{\beta'} & & \vertexU{z} \ar[lluu]_(0.25)\gamma}
\]
and $R = \{ (\alpha, \beta), (\beta, \gamma), (\alpha', \beta'),
(\beta', \gamma') \}$, then $\Delta'$ is the following quiver
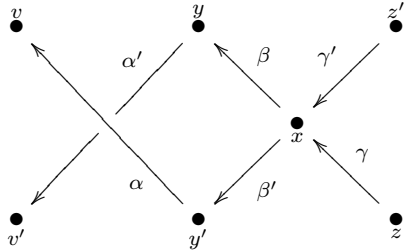
\[
\xymatrix{\vertexD{v} & & \vertexD{y}
\ar[lldd]_(0.25){\alpha'}|\hole & & \vertexD{z'} \ar[ld]_{\gamma'} \\
& & & \vertexU{x} \ar[lu]_\beta \ar[ld]^{\beta'} \\ \vertexU{v'} &
& \vertexU{y'} \ar[lluu]^(0.25)\alpha & & \vertexU{z}
\ar[lu]_\gamma}
\]
and $R' = \{ (\alpha, \beta'), (\beta', \gamma), (\alpha', \beta),
(\beta, \gamma') \}$ (in fact, the above example indicates all
possible changes which can appear). Then $k \bDelta' \simeq
\End_\Lambda (T)$, where
\[
T := \tau^{-1} S_x \oplus \bigoplus_{\substack{y \in \Delta_0 \\ y
\neq x}} \Lambda \cdot \bone_y,
\]
$S_x$ is the quotient of $\Lambda \cdot \bone_x$ modulo its unique
maximal submodule, and $\tau^{-1}$ is the quasi-inverse of the
Auslander--Reiten translation. Moreover, $T$ is a tilting module,
which we call the Brenner--Butler tilting (shortly, BB-tilting)
module at $x$.

We define coreflections and BB-cotilting modules (which are
tilting modules in the sense of our definition) dually. If
$\Lambda$ and $\Lambda'$ are gentle algebras, then we say that
$\Lambda$ and $\Lambda'$ are BB-equivalent if and only if there
exists a sequence $(\Lambda_0, \ldots, \Lambda_n)$, $n \in \bbN$,
of algebras such that $\Lambda_0 = \Lambda$, $\Lambda_n =
\Lambda'$, and for each $i \in [1, n]$ there exists a
BB-(co)tilting $\Lambda_{i - 1}$-module $T$ with $\End_{\Lambda_{i
- 1}} (T) \simeq \Lambda_i$. Obviously, if $\Lambda$ and
$\Lambda'$ are BB-equivalent, then $\Lambda$ and $\Lambda'$ are
derived equivalent. We will show that for the class of algebras we
consider these two notions coincide.

\section{The invariant of Avella-Alaminos and Geiss}
\label{section_AAG}

If $\bDelta$ is a gentle quiver, then there exist functions
$\sigma, \tau : \Delta_1 \to \{ \pm 1 \}$ such that the following
conditions are satisfied:
\begin{enumerate}

\item
if $\alpha, \beta \in \Delta_1$, $s \alpha = s \beta$ and $\alpha
\neq \beta$, then $\sigma \alpha = - \sigma \beta$,

\item
if $\alpha, \beta \in \Delta_1$, $t \alpha = t \beta$ and $\alpha
\neq \beta$, then $\tau \alpha = - \tau \beta$,

\item
if $\alpha, \beta \in \Delta_1$ and $s \alpha = t \beta$, then
$(\alpha, \beta) \in R$ if and only if $\sigma \alpha = \tau
\beta$.

\end{enumerate}
The functions $\sigma$ and $\tau$ are not uniquely determined by
$\bDelta$. From now on we always assume that given a gentle quiver
we are also given functions $\sigma$ and $\tau$ as above. If
$\bDelta$ is a gentle quiver and $\omega = (\alpha_1, \dots,
\alpha_n)$ is a path in $\Delta$ of positive length, then we put
$\sigma \omega := \sigma \alpha_n$ and $\tau \omega := \tau
\alpha_1$.

Now we fix a gentle quiver $\bDelta$. Following Avella-Alaminos
and Geiss \cite{AAG} we will define a function $f_\bDelta : \bbN^2
\to \bbN$, which will be also denoted $f_\Lambda$ provided
$\Lambda$ is (isomorphic to) the path algebra of $\bDelta$.

By a path in $\bDelta$ of positive length we mean a path
$(\alpha_1, \ldots, \alpha_n)$ in $\Delta$ of positive length such
that $(\alpha_i, \alpha_{i + 1}) \not \in R$ (equivalently,
$\sigma \alpha_i = - \tau \alpha_{i + 1}$) for each $i \in [1, n -
1]$. Moreover, for each $x \in \Delta_0$ we introduce two paths
$\bone_{x, 1}$ and $\bone_{x, -1}$ of length $0$ such that $s
\bone_{x, \varepsilon} := x =: t \bone_{x, \varepsilon}$, $\sigma
\bone_{x, \varepsilon} := \varepsilon$ and $\tau \bone_{x,
\varepsilon} := - \varepsilon$. A path $\omega$ in $\bDelta$ is
called maximal if there is no $\alpha \in \Delta_1$ such that $s
\alpha = t \omega$ and $\sigma \alpha = - \tau \omega$, and there
is no $\beta \in \Delta_1$ such that $t \beta = s \omega$ and
$\tau \beta = - \sigma \omega$ (such objects were called permitted
threads in~\cite{AAG}). By $\calM = \calM_{\bDelta}$ we denote the
set of all maximal paths in $\bDelta$.

By an antipath in $\bDelta$ of positive length we mean a path
$(\alpha_1, \ldots, \alpha_n)$ of positive length in $\Delta$ such
that $(\alpha_i, \alpha_{i + 1}) \in R$ (equivalently, $\sigma
\alpha_i = \tau \alpha_{i + 1}$) for each $i \in [1, n - 1]$.
Moreover, for each $x \in \Delta_0$ we introduce two antipaths
$\bone_{x, 1}'$ and $\bone_{x, -1}'$ of length $0$ such that $s
\bone_{x, \varepsilon}' := x =: t \bone_{x, \varepsilon}'$ and
$\sigma \bone_{x, \varepsilon}' := \varepsilon =: \tau \bone_{x,
\varepsilon}'$. An antipath $\omega$ is called maximal if there is
no $\alpha \in \Delta_1$ such that $s \alpha = t \omega$ and
$\sigma \alpha = \tau \omega$ and there is no $\beta \in \Delta_1$
such that $t \beta = s \omega$ and $\tau \beta = \sigma \omega$
(these objects correspond to forbidden threads in the terminology
of~\cite{AAG}). By $\calN = \calN_{\bDelta}$ we denote the set of
all maximal antipaths in $\bDelta$.

If $\omega \in \calM$, then there exists unique $\omega' \in
\calN$ such that $t \omega' = t \omega$ and $\tau \omega' = - \tau
\omega$. Moreover, the function $\phi_{\bDelta} : \calM \to \calN$
obtained in this way is a bijection. Similarly, we obtain a
bijection $\psi_{\bDelta} : \calN \to \calM$ by associating with
$\omega \in \calN$ the unique $\omega' \in \calM$ such that $s
\omega' = s \omega$ and $\sigma \omega' = - \sigma \omega$.
Finally, we define a bijection $\Phi_{\bDelta} : \calN \to \calN$
by $\Phi_{\bDelta} := \phi_{\bDelta} \circ \psi_{\bDelta}$. This
bijection induces an action of $\bbZ$ on $\calN$ and we denote by
$\calN / \bbZ$ the set of the orbits with respect to this action.
If $\calO \in \calN / \bbZ$, then we put $p (\calO) := |\calO|$
and $q (\calO) := \sum_{\omega \in \calO} \ell (\omega)$.

Let $\calC = \calC_{\bDelta}$ be the set of arrows $\alpha \in
\Delta_1$ such that $(\alpha)$ is not a subpath of a maximal
antipath in $\bDelta$. For each $\alpha \in \calC$ there exists
unique $\beta \in \calC$ such that $t \beta =  s \alpha$ and $\tau
\beta = \sigma \alpha$. In this way we obtain a bijection
$\Psi_{\bDelta} : \calC \to \calC$, which induces an action of
$\bbZ$ on $\calC$. If $\calO \in \calC / \bbZ$, then we put $p
(\calO) := 0$ and $q (\calO) := |\calO|$. In other words, if
$\calO = \bbZ \cdot \alpha$, then $q (\calO)$ is the minimal $q
\in \bbN_+$ such that there exists an antipath $(\alpha_0, \ldots,
\alpha_q)$ with $\alpha_0 = \alpha = \alpha_q$ (note that
$\alpha_i \neq \alpha_j$ for all $i, j \in [0, q - 1]$, $i \neq
j$, but it may happen that $s \alpha_i = s \alpha_j$ for some $i,
j \in [0, q - 1]$, $i \neq j$).

For $p, q \in \bbN$ we denote by $f_{\bDelta} (p, q)$ the number
of $\calO \in \calN / \bbZ \cup \calC / \bbZ$ such that $p (\calO)
= p$ and $q (\calO) = q$. Observe that $f_{\bDelta} (0, 3)$ counts
the number of the cycles
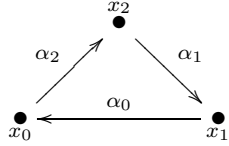
\[
\xymatrix{& \vertexD{x_2} \ar[rd]^{\alpha_1} \\ \vertexU{x_0}
\ar[ru]^{\alpha_2} & & \vertexU{x_1} \ar[ll]_{\alpha_0}}
\]
such that $\alpha_0 \neq \alpha_1 \neq \alpha_2 \neq \alpha_0$ and
$(\alpha_0, \alpha_1), (\alpha_1, \alpha_2), (\alpha_2, \alpha_0)
\in R$. We will call such configurations (more precisely, the
orbits in $\calC / \bbZ$ consisting of 3 arrows) triangles.
Observe that $x_0 \neq x_1 \neq x_2 \neq x_0$ in the above
situation, since otherwise we would have paths of arbitrary length
in $\bDelta$.

The following is the main result of~\cite{AAG}.

\begin{theorem} \label{theorem_AAG}
Let $\Lambda$ and $\Lambda'$ be gentle algebras. If $\Lambda$ and
$\Lambda'$ are derived equivalent, then $f_\Lambda =
f_{\Lambda'}$.
\end{theorem}

It is worth to observe the following.

\begin{lemma}
Let $\bDelta$ be a gentle quiver. Then
\[
\sum_{\calO \in \calN_{\bDelta} / \bbZ \cup \calC_{\bDelta} /
\bbZ} p (\calO) = 2 |\Delta_0| - |\Delta_1| \quad \text{and} \quad
\sum_{\calO \in \calN_{\bDelta} / \bbZ \cup \calC_{\bDelta} /
\bbZ} q (\calO) = |\Delta_1|.
\]
\end{lemma}

\begin{proof}
The latter observation is obvious, for the proof of the former we
first observe that
\[
\sum_{\calO \in \calN_{\bDelta} / \bbZ \cup \calC_{\bDelta} /
\bbZ} p (\calO) = \sum_{\calO \in \calN_{\bDelta} / \bbZ} p
(\calO) = |\calN_{\bDelta}| = |\calM_{\bDelta}|.
\]
Next, if $(x, \varepsilon) \in \Delta_0 \times \{ \pm 1 \}$, then
either there exists $\omega \in \calM_{\bDelta}$ such that $s
\omega = x$ and $\sigma \omega = \varepsilon$ or there exists
$\alpha \in \Delta_1$ such that $t \alpha = x$ and $\tau \alpha =
- \varepsilon$. One easily observes that in this way we may define
a bijection between $\Delta_0 \times \{ \pm 1 \}$ and
$\calM_{\bDelta} \cup \Delta_1$, which finishes the proof.
\end{proof}

Now we characterize, in terms of the above invariant, classes of
gentle quivers, which will play an important role in our
considerations. For $p, q \in \bbN$ we denote by $[p, q]$ the
characteristic function of the subset $\{ (p, q) \}$ of $\bbN^2$.

A gentle quiver $\bDelta$ is said to be of tree type if
$|\Delta_0| = |\Delta_1| + 1$. Recall~\cite{AH} that the gentle
quivers of tree type are precisely the gentle quivers whose path
algebras are derived equivalent to the path algebras of Dynkin
quivers of type $\bbA$.

\begin{lemma} \label{lemma_gentle_A}
If $\bDelta$ is a gentle quiver of tree type, then
\[
f_{\bDelta} = [|\Delta_0| + 1, |\Delta_0| - 1|].
\]
\end{lemma}

\begin{proof}
See~\cite{AAG}*{Section~7}.
\end{proof}

Consequently, we have the following characterization of the gentle
quivers of tree type.

\begin{proposition} \label{prop_gentle_A}
A gentle quiver $\bDelta$ is of tree type if and only if
$f_{\bDelta} = [p + 2, p]$ for some $p \in \bbN$.
\end{proposition}

\begin{proof}
If $f_{\bDelta} = [p + 2, p]$, then $|\Delta_1| = p$ and
$|\Delta_0| = \frac{1}{2} (|\Delta_0| + (p + 2)) = p + 1$, hence
$\bDelta$ is of tree type. The inverse implication follows from
the previous lemma.
\end{proof}

Let $\Delta$ be a connected quiver. An arrow $\alpha \in \Delta_1$
is called a branch arrow if the quiver $\Delta \setminus \{ \alpha
\}$ is not connected, otherwise we call $\alpha$ a cycle arrow.
Obviously $\Delta$ contains a cycle arrow if and only if
$|\Delta_0| \leq |\Delta_1|$. We say that $\Delta$ is a 1-cycle
quiver if $|\Delta_0| = |\Delta_1|$. If $\Delta$ is a 1-cycle
quiver and there are no branch arrows in $\Delta$, then we call
$\Delta$ a cycle. We will always assume that given a cycle
$\Delta$ we are also given its orientation, i.e., we may speak
about clockwise and anticlockwise oriented arrows and relations.

A gentle quiver $\bDelta$ is called a 1-cycle gentle quiver if
$\Delta$ is a 1-cycle quiver. We have the following
characterization of the 1-cycle gentle quivers.

\begin{proposition} \label{prop_1_cycle}
A gentle quiver $\bDelta$ is a 1-cycle gentle quiver if and only
if $f_{\bDelta} = [p + r, p] + [q - r, q]$ for some $p, q, r \in
\bbN$.
\end{proposition}

\begin{proof}
Analogous to the proof of Proposition~\ref{prop_gentle_A} (in
particular, we have to use results of~\cite{AAG}*{Section~7}).
\end{proof}

We present a combinatorial interpretation of the numbers in the
above proposition for a 1-cycle gentle quiver $\bDelta$ without
branch arrows (i.e., $\Delta$ is a cycle). Let $\Delta_1'$ and
$\Delta_1''$ denote the sets of the clockwise and the
anticlockwise oriented arrows, respectively. We divide the
clockwise oriented arrows into two classes $\Delta_1'^{(1)}$ and
$\Delta_1'^{(2)}$ in the following way:
\[
\Delta_1'^{(1)} := \{ \alpha \in \Delta_1' \mid \text{there exists
$\beta \in \Delta_1$ such that $(\beta, \alpha) \in R$} \}
\]
(i.e., $\Delta_1'^{(1)}$ consists of the clockwise oriented arrows
$\alpha$ such that $t \alpha$ is the middle vertex of a zero
relation) and $\Delta_1'^{(2)} := \Delta_1' \setminus
\Delta_1'^{(1)}$. Next for each $\alpha \in \Delta_1'$ we define
$\omega_\alpha \in \calN_\bDelta$: we put $\omega_\alpha :=
\bone_{t \alpha, - \tau \alpha}'$ if $\alpha \in \Delta_1'^{(1)}$
(note that $\omega_\alpha$ is the maximal antipath $\omega$ in
$\bDelta$ such that $t \omega = t \alpha$ and $\tau \omega = -
\tau \alpha$ in this case), and $\omega_\alpha$ is the maximal
antipath $\omega$ in $\bDelta$ such that $t \omega = t \alpha$ and
$\tau \omega = \tau \alpha$ if $\alpha \in \Delta_1'^{(2)}$. We
define the sets $\Delta_1''^{(1)}$ and $\Delta_1''^{(2)}$, and the
paths $\omega_\alpha$ for $\alpha \in \Delta_1''$, similarly.
Finally, we put
\[
\calO' := \{ \omega_\alpha \mid \alpha \in \Delta_1'^{(1)} \cup
\Delta_1''^{(2)} \} \quad \text{and} \quad  \calO'' := \{
\omega_\alpha \mid \alpha \in \Delta_1'^{(2)} \cup
\Delta_1''^{(1)} \}.
\]

\begin{proposition} \label{prop_interpretation}
Let $\bDelta$ be a 1-cycle gentle quiver without branch arrows.
Using the above notation we have the following.
\begin{enumerate}

\item
If $\calO' \neq \varnothing \neq \calO''$, then
\[
\calN_\bDelta / \bbZ = \{ \calO', \calO'' \} \qquad \text{and}
\qquad \calC_{\bDelta} / \bbZ = \varnothing.
\]

\item
If either $\calO' = \varnothing$ or $\calO'' = \varnothing$, then
\[
\calN_\bDelta / \bbZ = \{ \calO' \cup \calO'' \} \qquad \text{and}
\qquad \calC_{\bDelta} / \bbZ = \{ \Delta_1 \}.
\]

\end{enumerate}
In particular,
\[
f_{\bDelta} = [|\Delta_1'| - r, |\Delta_1'|] + [|\Delta_1''| + r,
|\Delta_1''|],
\]
where $r := |\Delta_1'^{(1)}| - |\Delta_1''^{(1)}|$ is the
difference between the numbers of the clockwise and the
anticlockwise oriented relations.
\end{proposition}

\begin{proof}
Observe that if $x \in \Delta_0$ and $\varepsilon \in \{ \pm 1
\}$, then there exists $\omega \in \calN_{\bDelta}$ such that $t
\omega = x$ and $\tau \omega = \varepsilon$ if and only if there
is no $\alpha \in \Delta_1$ such that $s \alpha = x$ and $\sigma
\alpha = \varepsilon$. This implies that $\calN_{\bDelta} = \calO'
\cup \calO''$. Indeed, if $\omega \in \calN_{\bDelta}$, then there
exists $\alpha \in \Delta_1$ such that $t \alpha = t \omega$
(otherwise, there exist $\beta', \beta'' \in \Delta_1$ such that
$\beta' \neq \beta''$ and $s \beta' = t \omega = s \beta''$, thus
either $\sigma \beta' = \tau \omega$ or $\sigma \beta'' = \tau
\omega$). Now, if there is $\alpha \in \Delta_1$ such that $t
\alpha = t \omega$ and $\tau \alpha = \tau \omega$, then $\alpha
\in \Delta_1'^{(2)} \cup \Delta_1''^{(2)}$ and $\omega =
\omega_\alpha$. In the other case, $\tau \alpha = - \tau \omega$
for unique $\alpha \in \Delta_1$ such that $t \alpha = t \omega$.
Since $\Delta$ is a cycle, there is unique $\beta \in \Delta_1$
such that $s \beta = t \omega$. The maximality of $\omega$ implies
that $\sigma \beta = - \tau \omega = \tau \alpha$, thus $\alpha
\in \Delta_1'^{(1)} \cup \Delta_1''^{(1)}$ and $\omega = \bone_{t
\alpha, - \tau \alpha} = \omega_\alpha$.

Observe that $\calC_{\bDelta} \neq \varnothing$ if and only if the
arrows of $\Delta$ form an oriented cycle such that $(\alpha,
\beta) \in R$ for all $\alpha, \beta \in \Delta_1$ with $s \alpha
= t \beta$. This means that $\calC_{\bDelta} \neq \varnothing$ if
and only if $\calC_{\bDelta} = \Delta_1$. Moreover, if this is the
case, then $\calC_{\bDelta} / \bbZ = \{ \Delta_1 \}$ and either
$\Delta_1 = \Delta_1'^{(1)}$ or $\Delta_1 = \Delta_1''^{(1)}$.

Finally, the formula for $\calN_{\bDelta} / \bbZ$ follows by an
analysis of the action of $\bbZ$ on $\calN_{\bDelta}$.
\end{proof}

A special role among the 1-cycle gentle quivers is played by the
gentle quivers of type $\tilde{\bbA}$, i.e., the gentle quivers
whose path algebras are derived equivalent to the path algebras of
Euclidean quivers of type $\tilde{\bbA}$. We have the following
characterization of the gentle quivers of type $\tilde{\bbA}$.

\begin{proposition} \label{prop_typeAtilde}
Let $\bDelta$ be 1-cycle gentle quiver. Then $\bDelta$ is of type
$\tilde{\bbA}$ if and only if $f_{\bDelta} = [p, p] + [q, q]$ for
some $p, q \in \bbN$.
\end{proposition}

\begin{proof}
See~\cite{AAG}*{Section~7}.
\end{proof}

\section{Completion procedure} \label{section_completion}

Let $\bDelta$ be a gentle quiver. We say that a relation $(\alpha,
\beta) \in R$ is isolated if $(\alpha, \beta) \in
\calN_{\bDelta}$. Given a set $R_0$ of isolated relations, we
define a quiver with relations $\bDelta'$, which we call the
quiver obtained from $\bDelta$ by completing the relations from
$R_0$, in the following way: $\Delta_0' := \Delta_0$, $\Delta_1'
:= \Delta_1 \cup \{ \gamma_\rho \mid \rho \in R_0 \}$ (where
$\gamma_\rho$, $\rho \in R_0$, are ``new'' arrows) and
\[
R' := R \cup \{ (\gamma_\rho, \alpha), (\beta, \gamma_\rho) \mid
\rho = (\alpha, \beta) \in R_0 \}.
\]

The following observation will be crucial.

\begin{lemma} \label{lemma_completion}
Let $\rho = (\alpha, \beta)$ be an isolated relation in a gentle
quiver $\bDelta$ and $\bDelta'$ be the quiver obtained from
$\bDelta$ by completing $\rho$. Then the following hold:
\begin{enumerate}

\item
$\bDelta'$ is a gentle quiver if and only if $\bbZ \cdot \rho \neq
\{ \rho \}$.

\item
If $\bbZ \cdot \rho \neq \{ \rho \}$, then $\calN_{\bDelta'} =
\calN_{\bDelta} \setminus \{ \rho \}$ and
\[
\calN_{\bDelta'} / \bbZ = \{ \calO' \in \calN_{\bDelta} / \bbZ
\mid \calO' \neq \bbZ \cdot \rho \} \cup \{ (\bbZ \cdot \rho)
\setminus \{ \rho \} \}.
\]
Moreover, $\calC_{\bDelta'} = \calC_{\bDelta} \cup \{ \alpha,
\beta, \gamma_\rho \}$ and
\[
\calC_{\bDelta'} / \bbZ = \calC / \bbZ \cup \{ \{ \alpha, \beta,
\gamma_\rho \} \}.
\]

\end{enumerate}
\end{lemma}

\begin{proof}
Let $\omega' := \psi_{\bDelta} \rho$ and $\omega'' :=
\phi_{\bDelta}^{-1} \rho$.

(1)~It is clear that the first three conditions of the definition
of  a gentle algebra are satisfied by $\bDelta'$. Moreover,
$\omega' \cdot (\gamma_\rho) \cdot \omega''$ is a path in
$\Delta$, which does not contain a subpath from $R$, hence the
last condition is satisfied if and only if $\omega' \neq
\omega''$, that is, if and only if $\bbZ \cdot \rho \neq \{ \rho
\}$.

(2)~Assume that $\bbZ \cdot \rho \neq \{ \rho \}$. The equalities
\[
\calN_{\bDelta'} = \calN_{\bDelta} \setminus \{ \rho \} \qquad
\text{and} \qquad \calC_{\bDelta'} = \calC_{\bDelta} \cup \{
\alpha, \beta, \gamma_\rho \}
\]
are immediate. Moreover,
\[
\Psi_{\bDelta'} \alpha = \beta, \qquad \Psi_{\bDelta'} \beta =
\gamma_\rho \qquad \text{and} \qquad \Psi_{\bDelta'} \gamma_\rho =
\alpha,
\]
while $\Psi_{\bDelta'} \omega = \Psi_{\bDelta} \omega$ for all
$\omega \in \calC_{\bDelta}$. Observe that our assumption implies
that $\omega' \cdot (\gamma_\rho) \cdot \omega'' \in
\calM_{\bDelta'}$. Consequently,
\[
\Phi_{\bDelta'} \omega =
\begin{cases}
\Phi_{\bDelta} \rho & \omega = \Phi_{\bDelta}^{-1} \rho,
\\
\Phi_{\bDelta} \omega & \text{otherwise},
\end{cases}
\]
which finishes the proof.
\end{proof}

Let $\Delta$ the following quiver
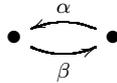
\[
\xymatrix{\bullet \ar@/_/[r]_\beta & \bullet \ar@/_/[l]_\alpha}
\]
and $R := \{ (\alpha, \beta) \}$. Then $\rho := (\alpha, \beta)$
is an isolated relation in $\bDelta$. If $\bDelta'$ is obtained
from $\bDelta$ by completing $\rho$, then $\Delta'$ is the
following quiver
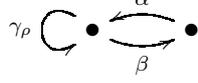
\[
\xymatrix{\bullet \ar@/_/[r]_\beta \ar@(ul,dl)_{\gamma_\rho}&
\bullet \ar@/_/[l]_\alpha }
\]
and $R' = \{ (\alpha, \beta), (\gamma_\rho, \alpha), (\beta,
\gamma_\rho) \}$. In particular, $\bDelta'$ is not gentle.

We list some consequences of the above lemma.

\begin{corollary}
Let $R_0$ be a set of isolated relations in a gentle quiver
$\bDelta$ and $\bDelta'$ be the quiver obtained from $\bDelta$ by
completing the relations from $R_0$. Then $\bDelta'$ is a gentle
quiver if and only if $\calO \not \subset R_0$ for each $\calO \in
\calN_{\bDelta} / \bbZ$.
\end{corollary}

\begin{proof}
Immediate.
\end{proof}

\begin{corollary} \label{coro_f_completion}
Let $R_0$ be a set of isolated relations in a gentle quiver
$\bDelta$ and $\bDelta'$ be the quiver obtained from $\bDelta$ by
completing the relations from $R_0$. If $\bDelta'$ is a gentle
quiver, then
\[
f_{\bDelta'} = |R_0| \cdot [0, 3] + \sum_{\calO \in
\calN_{\bDelta} / \bbZ \cup \calC_{\bDelta} / \bbZ} [p (\calO) - m
(\calO), q (\calO) - 2 m (\calO)],
\]
where $m (\calO) := |R_0 \cap \calO|$ for $\calO \in
\calN_{\bDelta} / \bbZ \cup  \calC_{\bDelta} / \bbZ$.
\end{corollary}

\begin{proof}
Follows from the above lemma by induction.
\end{proof}

Now we study the inverse operation.

If $\Delta$ is a quiver and $\Delta_1' \subset \Delta_1$, then we
put $\Delta \setminus \Delta_1' := (\Delta_0, \Delta_1 \setminus
\Delta_1')$. Similarly, if $\bDelta$ is a gentle quiver and
$\Delta_1' \subset \Delta_1$, then we denote by $\bDelta \setminus
\Delta_1'$ the pair
\[
(\Delta \setminus \Delta_1', \{ \rho \in R : \text{$(\alpha)$ is
not a subpath of $\rho$ for each $\alpha \in \Delta_1'$} \}).
\]
In the above situation we say that $\bDelta \setminus \Delta_1'$
is obtained from $\bDelta$ by removing the arrows from
$\Delta_1'$.

Recall that by a triangle in a gentle quiver $\bDelta$ we mean
every orbit in $\calC_{\bDelta} / \bbZ$ consisting of 3 elements.
Obviously, if $\calO'$ and $\calO''$ are different triangles in a
gentle quiver $\bDelta$, then $\calO' \cap \calO'' = \varnothing$.

\begin{corollary} \label{prop_f_remove}
Let $\calO_1$, \ldots, $\calO_m$ be pairwise different triangles
in a gentle quiver $\bDelta$ and $\alpha_i \in \calO_i$ for each
$i \in [1, m]$. If $\bDelta' := \bDelta \setminus \{ \alpha_i \mid
i \in [1, m] \}$ and
\[
f_{\bDelta} = m \cdot [0, 3] + \sum_{i \in [1, n]} [p_i, q_i]
\]
for some $p_i, q_i \in \bbN$, $i \in [1, n]$, then $\bDelta'$ is a
gentle quiver and
\[
f_{\bDelta'} = \sum_{i \in [1, n]} [p_i + m_i, q_i + 2 m_i]
\]
for some $m_1, \ldots, m_n \in \bbN$ such that $m_1 + \cdots + m_n
= m$.
\end{corollary}

\begin{proof}
One easily checks that $\bDelta'$ is a gentle quiver. Now, for
each $i \in [1, n]$ let $(\beta_i, \gamma_i)$ be a path in
$\Delta$ such that $\calO_i = \{ \alpha_i, \beta_i, \gamma_i \}$.
Then $\bDelta$ is (isomorphic to) the quiver obtained from
$\bDelta'$ by completing the relations $(\beta_i, \gamma_i)$, $i
\in [1, n]$, so the claim follows immediately from the previous
corollary.
\end{proof}

\section{Cluster tilted algebras and the main results}
\label{section_cluster}

For an acyclic quiver $Q$, a category called the cluster category
was defined in~\cite{BMRRT}. Let $\tau_D^{-1}$ be the
quasi-inverse of the Auslander--Reiten translation in the bounded
derived category $D^b (k Q)$. Then the cluster category $\calC =
\calC_{k Q}$ is the orbit category $D^b (k Q) / F$, where $F$ is
the autoequivalence $\tau_D^{-1} \Sigma$. Cluster categories are
canonically triangulated, as shown in~\cite{K}.

An object $T$ in $\calC$ with $\Ext^1_{\calC} (T,T) = 0$ and
$|Q_0|$ pairwise non-isomorphic indecomposable summands, is called
a cluster tilting object. There is a natural embedding of $\mod k
Q$ into $\calC_{k Q}$. Under this embedding, tilting modules are
mapped to tilting objects. The endomorphism ring $\End_{\calC}
(T)$ of a tilting object, is a cluster tilted algebra of type
$Q$~\cite{BMR}. Here we will consider cluster tilted algebras of
types $\bbA$ and $\tilde{\bbA}$. They appear in the following
theorem~\cite{ABChJP}*{Theorem~3.3}.

\begin{theorem} \label{theorem_ABChJP}
Let $C$ be a cluster tilted algebra of type $Q$. Then $C$ is
gentle if and only if $Q$ is either of Dynkin type $\bbA$ or of
Eulidean type $\tilde{\bbA}$.
\end{theorem}

We have the following consequence for our class of algebras.

\begin{corollary} \label{corollary_ABChJP}
Let $C$ be a cluster tilted algebra of type $Q$. If $C$ is derived
equivalent to a gentle algebra, then $Q$ is either of Dynkin type
$\bbA$ or of Eulidean type $\tilde{\bbA}$.
\end{corollary}

\begin{proof}
\cite{SZ}*{Corollary~1.2} implies that $C$ is gentle, hence the
claim follows from the previous theorem.
\end{proof}

Now we collect facts about cluster tilted gentle algebras.

The following theorem is a reformulation
of~\cite{BV}*{Proposition~3.1}.

\begin{proposition} \label{proposition_cluster_A}
An algebra $\Lambda$ is a cluster tilted algebra of type $\bbA$ if
and only if there exists a gentle quiver $\bDelta$ of tree type
such that $R$ consists of isolated relations and $\Lambda$ is
isomorphic to the path algebra of the quiver obtained from
$\bDelta$ by completing the relations from $R$.
\end{proposition}

As an immediate consequence of the above proposition,
Corollary~\ref{coro_f_completion} and
Proposition~\ref{prop_gentle_A} we obtain the following.

\begin{corollary} \label{corollary_cluster_A}
If $\bDelta$ is a gentle quiver such that its path algebra is
derived equivalent to a cluster tilted algebra of type $\bbA$,
then
\[
f_{\bDelta} = m \cdot [0, 3] + [p + m + 2, p]
\]
for some $m, p \in \bbN$.
\end{corollary}

Moreover, the following result follows from the proof
of~\cite{BV}*{Theorem}.

\begin{theorem} \label{theorem_BB_A}
Let $\Lambda$ and $\Lambda'$ be cluster tilted algebras of type
$\bbA$. Then $\Lambda$ and $\Lambda'$ are derived equivalent if
and only if $\Lambda$ and $\Lambda'$ are BB-equivalent.
\end{theorem}

\begin{proof}
It is enough to observe that all derived equivalences used
in~\cite{BB} are in fact (co)reflections as defined in
Section~\ref{section_tilting}.
\end{proof}

We have the following analogue of
Proposition~\ref{proposition_cluster_A} for the cluster tilted
algebras of type $\tilde{\bbA}$.

\begin{proposition} \label{proposition_cluster_A_tilde}
The following conditions are equivalent for a gentle algebra
$\Lambda$.
\begin{enumerate}

\item
$\Lambda$ is a cluster tilted algebra of type $\tilde{\bbA}$.

\item
There exists a gentle quiver $\bDelta$ of type $\tilde{\bbA}$ such
that $R$ consists of isolated relations and $\Lambda$ is
isomorphic to the path algebra of the quiver obtained from
$\bDelta$ by completing the relations from $R$.

\item
There exists a 1-cycle gentle quiver $\bDelta$ such that $R$
consists of isolated relations and $\Lambda$ is isomorphic to the
path algebra of the quiver obtained from $\bDelta$ by completing
the relations from $R$.

\end{enumerate}
\end{proposition}

\begin{proof}
Follows from~\cite{ABSch} (and calculations
in~\cite{ABChJP}*{Section~3}).

(2)$\implies$(3) Obvious.

(3)$\implies$(1) Let $\bDelta'$ be the quiver obtained from
$\bDelta$ by completing the relations from $R$. Then $\Delta'$
belongs to the class of quivers considered
in~\cite{Ba}*{Section~3}, which implies that $\Lambda$ is cluster
tilted of type $\tilde{\bbA}$.
\end{proof}

We warn the reader that if $\bDelta$ is a 1-cycle gentle quiver
such that $R$ consists of isolated relations, then the quiver
obtained from $\bDelta$ by completing the relations from $R$ may
not be gentle, as the example from
Section~\ref{section_completion} shows.

Again, we have the following immediate consequence of the above
proposition, Proposition~\ref{prop_typeAtilde}, and
Corollary~\ref{coro_f_completion}.

\begin{corollary} \label{corollary_cluster_A_tilde}
If $\bDelta$ is a gentle quiver such that its path algebra is
derived equivalent to a cluster tilted algebra of type
$\tilde{\bbA}$, then
\[
f_{\bDelta} = (m_1 + m_2) \cdot [0, 3] + [p + m_1, p] + [q + m_2,
q]
\]
for some $m_1, m_2, p, q \in \bbN$ such that $p + m_1 > 0$ and $q
+ m_2 > 0$.
\end{corollary}

Finally, the following theorem is a consequence of~\cite{Ba}*{the
proof of Theorem~5.5}.

\begin{theorem} \label{theorem_BB_A_tilde}
Let $\Lambda$ and $\Lambda'$ be cluster tilted algebras of type
$\tilde{\bbA}$. Then the following conditions are equivalent:
\begin{enumerate}

\item
$\Lambda$ and $\Lambda'$ are derived equivalent.

\item
$\Lambda$ and $\Lambda'$ are BB-equivalent.

\item
$f_\Lambda = f_{\Lambda'}$.

\end{enumerate}
\end{theorem}

\begin{proof}
For the proof of implication (1)$\implies$(2) one has to check
again that all derived equivalences used in~\cite{Ba} are
(co)reflections. Next, implication (2)$\implies$(3) is obvious.
Finally, in the proof of~\cite{Ba}*{Theorem~5.5} the author shows
that $f_\Lambda \neq f_{\Lambda'}$ for cluster tilted algebras
$\Lambda$ and $\Lambda'$ of type $\tilde{\bbA}$ which are not in
the same derived equivalence class, hence the implication
(3)$\implies$(1) follows.
\end{proof}

Now we formulate the main results of the paper.

\begin{maintheorem} \label{main_A}
Let $\Lambda$ be a gentle algebras. Then $\Lambda$ is derived
equivalent to a cluster tilted algebra of type $\bbA$ if and only
if
\[
f_{\bDelta} = m \cdot [0, 3] + [p + m + 2, p]
\]
for some $m, p \in \bbN$.
\end{maintheorem}

\begin{maintheorem} \label{main_B}
Let $\Lambda$ be a gentle algebras. Then $\Lambda$ is derived
equivalent to a cluster tilted algebra of type $\tilde{\bbA}$ if
and only if
\[
f_{\bDelta} = (m_1 + m_2) \cdot [0, 3] + [p + m_1, p] + [q + m_2,
q]
\]
for some $m_1, m_2, p, q \in \bbN$ such that $p + m_1 > 0$ and $q
+ m_2 > 0$.
\end{maintheorem}

Taking into account Corollary~\ref{corollary_ABChJP} we
immediately get the following.

\begin{maincorollary}
Let $\Lambda$ be a gentle algebras. Then $\Lambda$ is derived
equivalent to a cluster tilted algebra if and only if either
\[
f_{\bDelta} = m \cdot [0, 3] + [p + m + 2, p]
\]
for some $m, p \in \bbN$, or
\[
f_{\bDelta} = (m_1 + m_2) \cdot [0, 3] + [p + m_1, p] + [q + m_2,
q]
\]
for some $m_1, m_2, p, q \in \bbN$ such that $p + m_1 > 0$ and $q
+ m_2 > 0$.
\end{maincorollary}

Moreover, we show the following.

\begin{maintheorem} \label{main_D}
Let $\Lambda$ and $\Lambda'$ be gentle algebras derived equivalent
to cluster tilted algebras. Then the following conditions are
equivalent:
\begin{enumerate}

\item
$\Lambda$ and $\Lambda'$ are derived equivalent.

\item
$\Lambda$ and $\Lambda'$ are BB-equivalent.

\item
$f_\Lambda = f_{\Lambda'}$.

\end{enumerate}
\end{maintheorem}

Recall that an algebra $\Lambda$ is called Gorenstein if the
injective dimensions of $\Lambda$ both as a left and as a right
$\Lambda$-module are finite. If this is the case, then these
dimensions coincide~\cite{Hap}*{Lemma~1.2}, and this common value
is called the Gorenstein dimension $\Gdim \Lambda$ of $\Lambda$.
Both the gentle and the cluster tilted algebras are
Gorenstein~\cites{GR, KR}. Moreover, in the case of cluster tilted
algebras we have the following~\cite{KR}.

\begin{theorem} \label{theo_Keller_Reiten}
Let $\Lambda$ be a cluster tilted algebra. Then $\Gdim \Lambda
\leq 1$.
\end{theorem}

The above property characterizes the gentle cluster tilted
algebras.

\begin{maintheorem} \label{main_E}
Let $\Lambda$ be a gentle algebra derived equivalent to a cluster
tilted algebra. Then $\Lambda$ is cluster titled if and only if
$\Gdim \Lambda \leq 1$.
\end{maintheorem}

\section{Proof of Theorem~\ref{main_A}} \label{section_A}

Let $\calA$ denote the class of the gentle quivers $\bDelta$ such
that
\[
f_{\bDelta} = m \cdot [0, 3] + [p + m + 2, p]
\]
for some $m, p \in \bbN$. Taking into account
Corollary~\ref{corollary_cluster_A}, the following proposition
will imply Theorem~\ref{main_A}.

\begin{proposition} \label{proposition_A}
If $\bDelta \in \calA$, then $k \bDelta$ is BB-equivalent to a
cluster tilted algebra of type $\bbA$.
\end{proposition}

The following observation will be used many times in our
considerations without mentioning it explicitly: if $\bDelta \in
\calA$, then every orbit in $\calC_{\bDelta} / \bbZ$ is a
triangle.

We start with the following lemma.

\begin{lemma} \label{lemma_BV}
Let $\bDelta \in \calA$ and $\calO_1$, \ldots, $\calO_m$ be the
pairwise different triangles in $\bDelta$. If $\alpha_i \in
\calO_i$ for each $i \in [1, m]$ and $\bDelta' := \bDelta
\setminus \{ \alpha_i \mid i \in [1, m] \}$, then $\bDelta'$ is a
gentle quiver of tree type.
\end{lemma}

\begin{proof}
Follows immediately from Propositions~\ref{prop_f_remove}
and~\ref{prop_gentle_A}.
\end{proof}

The quiver $\bDelta'$ described in the above lemma will be called
a model of $\bDelta$. Obviously, a model of $\bDelta$ is not
uniquely determined by $\bDelta$.

If $\bDelta$ is a gentle quiver, then we call $(\alpha, \beta) \in
R$ a branch relation if $\alpha$ or $\beta$ is a branch arrow.
Observe that if $\bDelta \in \calA$, then $(\alpha, \beta) \in R$
is a branch relation if and only if both $\alpha$ and $\beta$ are
branch arrows.

We have the following description of the branch and the cycle
arrows for the gentle quivers from $\calA$.

\begin{lemma} \label{lemma_cyclebranch}
Let $\bDelta \in \calA$ and $\alpha \in \Delta_1$. Then $\alpha$
is a branch \textup{(}cycle\textup{)} arrow if and only if $\alpha
\not \in \calC_{\bDelta}$ \textup{(}$\alpha \in \calC_{\bDelta}$,
respectively\textup{)}.
\end{lemma}

\begin{proof}
Obviously, $\alpha$ is a cycle arrow if $\alpha \in
\calC_{\bDelta}$. On the other hand, if $\alpha \not \in
\calC_{\bDelta}$ and $\bDelta'$ is a model of $\bDelta$, then
$\alpha \in \Delta_1'$. Since $|\Delta_0'| = |\Delta_1'| + 1$
(according to the previous lemma), $\Delta' \setminus \{ \alpha
\}$ is not connected. This immediately implies that $\Delta
\setminus \{ \alpha \}$ is not connected and finishes the proof.
\end{proof}

As a consequence we obtain the following characterization of the
cluster tilted algebras of type $\bbA$.

\begin{corollary}
Let $\bDelta \in \calA$. Then the path algebra of  $\bDelta$ is a
cluster tilted algebra of type $\bbA$ if and only if there are no
branch relations in $\bDelta$.
\end{corollary}

\begin{proof}
Obviously, if the path algebra of $\bDelta$ is a cluster tilted
algebra of type $\bbA$, then there are no branch relations in
$\bDelta$ by Proposition~\ref{proposition_cluster_A}.

On the other hand, assume that there are no branch relations in
$\bDelta$. If $\bDelta'$ is a model of $\bDelta$, then the
previous lemma implies that $R'$ consists of isolated relations
and $\bDelta$ is obtained from $\bDelta'$ by completing the
relations from $R'$. This finishes the proof according to
Proposition~\ref{proposition_cluster_A}.
\end{proof}

Consequently, in order to prove Proposition~\ref{proposition_A} we
only need to show the following.

\begin{proposition} \label{prop_no_branch_rel_A}
Let $\bDelta \in \calA$. Then there exists a gentle quiver
$\bDelta'$ without branch relations such that $k \bDelta$ and $k
\bDelta'$ are BB-equivalent.
\end{proposition}

\begin{proof}
For a branch relation $(\alpha, \beta) \in R$ let $n_{(\alpha,
\beta)}^{\bDelta}$ be the number of the vertices in the connected
component of $\Delta \setminus \{ \alpha \}$ containing $t
\alpha$. Let $r_{\bDelta}$ denote the number of the branch
relations in $\bDelta$ and
\[
n_{\bDelta} := \min \{ n_\rho^{\bDelta} \mid \text{$\rho \in R$ is
a branch relation} \}
\]
(by convention $\min \varnothing = \infty$).

If $r_{\bDelta} > 0$, then we fix a branch relation $(\alpha,
\beta)$ such that $n_{(\alpha, \beta)}^{\bDelta} = n_{\bDelta}$.
Observe that we can apply the reflection at $t \alpha$. Indeed, if
$\delta \in \Delta_1$ and $s \delta = t \alpha$, then $(\delta,
\alpha) \not \in R$, since otherwise $n_{(\delta,
\alpha)}^{\bDelta} < n_{(\alpha, \beta)}^{\bDelta} = n_{\bDelta}$,
which contradicts the minimality of $n_{(\alpha,
\beta)}^{\bDelta}$. If $\bDelta'$ is the quiver obtained from
$\bDelta$ by applying the reflection at $t \alpha$, then we will
show that either $r_{\bDelta'} < r_{\bDelta}$ or $r_{\bDelta'} =
r_{\bDelta}$ and $n_{\bDelta'} < n_{\bDelta}$. Consequently, the
claim follows by induction.

We have to consider three cases:
\begin{enumerate}

\item
there exists a cycle arrow $\gamma$ such that $t \gamma = t
\alpha$,

\item
there exists a branch arrow $\gamma$ such that $\gamma \neq
\alpha$ and $t \gamma = t \alpha$,

\item
there is no arrow $\gamma$ such that $\gamma \neq \alpha$ and $t
\gamma = t \alpha$.

\end{enumerate}

(1)~First assume that there exists a cycle arrow $\gamma$ such
that $t \gamma = t \alpha$. Then $\gamma \in \calC_{\bDelta}$
according to Lemma~\ref{lemma_cyclebranch}. Let $\gamma' :=
\Psi_{\bDelta} \gamma$ and $\gamma'' := \Psi_{\bDelta} \gamma'$.
Observe that $s \gamma'' = t \alpha$. Moreover, if $s \delta = t
\alpha$ for some $\delta \in \Delta_1$, then $\delta = \gamma''$.
Indeed, if $\delta \neq \gamma''$, then $(\delta, \alpha) \in R$,
hence $\delta$ is a branch arrow, but this contradicts the
minimality of $n_{(\alpha, \beta)}^{\bDelta}$. Consequently,
$\Delta'$ is obtained from $\Delta$ by replacing the subquiver
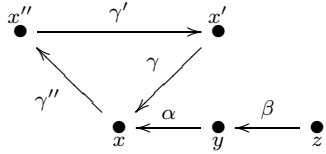
\[
\xymatrix{\vertexD{x''} \ar[rr]^{\gamma'} & & \vertexD{x'}
\ar[ld]_\gamma \\ & \vertexU{x} \ar[lu]^{\gamma''} & \vertexU{y}
\ar[l]_\alpha & \vertexU{z} \ar[l]_\beta}
\]
by the quiver
\[
\vcenter
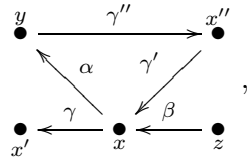
{\xymatrix{\vertexD{y} \ar[rr]^{\gamma''} & &
\vertexD{x''} \ar[ld]_{\gamma'} \\ \vertexU{x'} & \vertexU{x}
\ar[lu]_\alpha \ar[l]_\gamma & \vertexU{z} \ar[l]_\beta}},
\]
and
\[
R' = (R \setminus \{ (\alpha, \beta), (\gamma, \gamma'),
(\gamma'', \gamma) \}) \cup \{ (\alpha, \gamma'), (\gamma, \beta),
(\gamma'', \alpha) \}.
\]
Thus $r_{\bDelta'} = r_{\bDelta}$ in this case. Moreover,
$(\gamma, \beta)$ is a branch relation in $\bDelta'$ and
$n_{(\gamma, \beta)}^{\bDelta'} < n_{(\alpha, \beta)}^{\bDelta}$,
hence $n_{\bDelta'} < n_{\bDelta}$.

(2)~Next, assume that there is a branch arrow $\gamma$ such that
$\gamma \neq \alpha$ and $t \gamma = t \alpha$. By the minimality
of $n_{(\alpha, \beta)}^{\bDelta}$ there is no $\delta \in
\Delta_1$ such that $s \delta = t \alpha$ (as otherwise either
$(\delta, \alpha) \in R$ or $(\delta, \gamma) \in R$). If there is
no $\gamma' \in \Delta_1$ such that $(\gamma, \gamma') \in R$,
then $\Delta'$ is obtained from $\Delta$ by replacing the
subquiver $\xymatrix{\vertexU{y'} \ar[r]^\delta & \vertexU{x} &
\vertexU{y} \ar[l]_\alpha & \vertexU{z} \ar[l]_\beta}$ by the
quiver
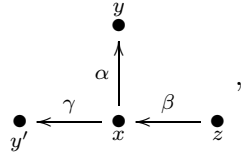
\[
\vcenter{\xymatrix{& \vertexD{y} \\ \vertexU{y'} & \vertexU{x}
\ar[l]_\gamma \ar[u]^\alpha & \vertexU{z} \ar[l]_\beta}},
\]
and $R' = (R \setminus \{ (\alpha, \beta) \}) \cup \{ (\delta,
\beta) \}$ (thus $r_{\bDelta'} = r_{\bDelta}$ and $n_{\bDelta'} <
n_{\bDelta}$). Otherwise, $\Delta'$ is obtained from $\Delta$ by
replacing the subquiver
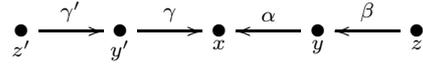
\[
\xymatrix{\vertexU{z'} \ar[r]^{\gamma'} & \vertexU{y'}
\ar[r]^\gamma & \vertexU{x} & \vertexU{y} \ar[l]_\alpha &
\vertexU{z} \ar[l]_\beta}
\]
by the quiver
\[
\vcenter
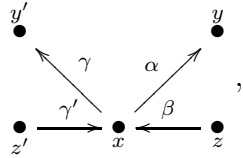
{\xymatrix{\vertexD{y'} & & \vertexD{y} \\ \vertexU{z'}
\ar[r]^{\gamma'} & \vertexU{x} \ar[lu]_\gamma \ar[ru]^\alpha &
\vertexU{z} \ar[l]_\beta}},
\]
and
\[
R' = (R \setminus \{ (\alpha, \beta), (\gamma, \gamma') \}) \cup
\{ (\alpha, \gamma'), (\gamma, \beta) \}.
\]
Obviously $r_{\bDelta} = r_{\bDelta'}$. Moreover, $n_{(\gamma,
\beta)}^{\bDelta'} < n_{(\alpha, \beta)}^{\bDelta}$ (as $x$ and
$z$ are not in the connected component of $\Delta' \setminus \{
\gamma \}$ containing $y'$), so $n_{\bDelta'} < n_{\bDelta}$.

(3)~Finally assume that there is no arrow $\gamma$ such that
$\gamma \neq \alpha$ and $t \gamma = t \alpha$. If in addition,
there is no arrow $\delta$ such that $s \delta = t \alpha$, then
$\Delta'$ is obtained from $\Delta$ by replacing the subquiver
$\xymatrix{\vertexU{x} & \vertexU{y} \ar[l]_\alpha & \vertexU{z}
\ar[l]_\beta}$ by the subquiver $\xymatrix{\vertexU{y} &
\vertexU{x} \ar[l]_\alpha & \vertexU{z} \ar[l]_\beta}$ and $R' = R
\setminus \{ (\alpha, \beta) \}$, hence $r_{\bDelta'} <
r_{\bDelta}$. On the other hand, if there is $\delta \in \Delta_1$
such that $s \delta = t \alpha$, then $\delta$ is a branch arrow
(otherwise, there would be a cycle arrow $\gamma$ such that $t
\gamma = t \alpha$). Moreover, by the minimality of $n_{(\alpha,
\beta)}^{\bDelta}$, we have that $(\delta, \alpha) \not \in R$ and
there is no $\delta' \in \Delta_1$ such that $s \delta' = t
\alpha$ and $\delta' \neq \delta$. Thus $\Delta'$ is obtained from
$\Delta$ by replacing the subquiver $\xymatrix{\vertexU{x'} &
\vertexU{x} \ar[l]_\delta & \vertexU{y} \ar[l]_\alpha &
\vertexU{z} \ar[l]_\beta}$ by the subquiver
$\xymatrix{\vertexU{x'} & \vertexU{y} \ar[l]_\gamma & \vertexU{x}
\ar[l]_\alpha & \vertexU{z} \ar[l]_\beta}$ and $R' = (R \setminus
\{ (\alpha, \beta) \}) \cup \{ (\gamma, \alpha) \}$. Consequently,
$r_{\bDelta'} = r_{\bDelta}$ and $n_{\bDelta'} < n_{\bDelta}$.
This finishes the proof.
\end{proof}

\section{Proof of Theorem~\ref{main_B}} \label{section_B}

Let $\tilde{\calA}$ denote the class of the gentle quivers
$\bDelta$ such that
\[
f_{\bDelta} = (m_1 + m_2) \cdot [0, 3] + [p + m_1, p] + [q + m_2,
q]
\]
for some $m_1, m_2, p, q \in \bbN$ such that $p + m_1 > 0$ and $q
+ m_2 > 0$. Taking into account
Corollary~\ref{corollary_cluster_A_tilde}, the following
proposition will imply Theorem~\ref{main_B}.

\begin{proposition} \label{proposition_A_tilde}
If $\bDelta \in \tilde{\calA}$, then $k \bDelta$ is BB-equivalent
to a cluster tilted algebra of type $\tilde{\bbA}$.
\end{proposition}

Observe that similarly as in the $\bbA$-case, every orbit in
$\calC_{\bDelta} / \bbZ$ is a triangle provided $\bDelta \in
\tilde{\calA}$.

We start the proof with the following lemma.

\begin{lemma} \label{lemma_erase_A_tilde}
Let $\bDelta \in \tilde{\calA}$ and $\calO_1$, \ldots, $\calO_m$
be the pairwise different triangles in $\bDelta$. If $\alpha_i \in
\calO_i$ for each $i \in [1, m]$ and $\bDelta' := \bDelta
\setminus \{ \alpha_i \mid i \in [1, m] \}$, then $\bDelta'$ is a
1-cycle gentle quiver.
\end{lemma}

\begin{proof}
Follows immediately from Propositions~\ref{prop_f_remove}
and~\ref{prop_1_cycle}.
\end{proof}

Again, we call the quiver $\bDelta'$ described in the above lemma
a model of $\bDelta$.

Let $\bDelta \in \tilde{\calA}$ and $\calO \in \calC_{\bDelta} /
\bbZ$. We say that $\calO$ is a branch triangle if for each subset
$\calR$ of $\calO$ such that $|\calR| = 2$ we have that $\Delta
\setminus \calR$ is not connected. Otherwise, we call $\calO$ a
cycle triangle. A cycle arrow $\alpha$ is called a strongly cycle
arrow if either $\alpha \in \Delta_1 \setminus \calC$ or $\alpha$
belongs to a cycle triangle. Recall that a vertex $x$ is said to
be adjacent to an arrow $\alpha$ if either $x = s \alpha$ or $x =
t \alpha$. Similarly, we say that a vertex $x$ is adjacent to a
triangle $\calO$ if there exists $\alpha \in \calO$ such that $x$
is adjacent to $\alpha$. A vertex $x$ will be called a strongly
cycle vertex if it is adjacent to a strongly cycle arrow. Observe
that every branch arrow in $\Delta$ belongs to $\Delta_1 \setminus
\calC_{\bDelta}$.

We have the following characterization of branch/cycle
arrows/trian\-gles.

\begin{lemma}
Let $\bDelta \in \tilde{\calA}$ and $\bDelta'$ be a model of
$\bDelta$. Then the following hold.
\begin{enumerate}

\item
$\Delta_1 \setminus \calC_{\bDelta} \subset \Delta_1'$.

\item
If $\alpha \in \Delta_1 \setminus \calC_{\bDelta}$, then $\alpha$
is a branch \textup{(}cycle\textup{)} arrow in $\Delta$ if and
only if $\alpha$ is a branch \textup{(}cycle,
respectively\textup{)} arrow in $\Delta'$.

\item
Let $\calO \in \calC_{\bDelta} / \bbZ$. Then $\calO$ is a branch
triangle if and only if $\calO \cap \Delta_1'$ consists of branch
arrows in $\Delta'$.

\item
Let $\calO \in \calC_{\bDelta} / \bbZ$. Then $\calO$ is a cycle
triangle if and only if $\calO \cap \Delta_1'$ contains a cycle
arrow in $\Delta'$.

\end{enumerate}
\end{lemma}

\begin{proof}
The above claims follow directly from the appropriate definitions.
\end{proof}

Let $\bDelta \in \tilde{\calA}$. One of the consequences of the
above fact is that the subquiver of $\Delta$ generated by the
strongly cycle arrows is connected (this follows since the
subquiver of a 1-cycle quiver generated by the cycle arrows is
connected). Consequently, given a branch arrow $\alpha$ there
exists a component of $\Delta \setminus \{ \alpha \}$ which does
not contain strongly cycle vertices (and this component is
obviously unique). We define $n_\alpha^{\bDelta}$ to be the number
of the vertices in this component. We put $n_\alpha^{\bDelta} :=
\infty$ if $\alpha$ is a cycle arrow. If $(\alpha, \beta)$ is a
branch relation, then we define $n_{(\alpha, \beta)}^{\bDelta} :=
\min \{ n_\alpha, n_\beta \}$. Observe the following: if $(\alpha,
\beta)$ is a branch relation and $n_{(\alpha, \beta)}^{\bDelta} =
n_\alpha^{\bDelta}$, then $t \alpha$ is not a strongly cycle
vertex (if $t \alpha$ is a strongly cycle vertex, then one easily
shows that $n_\beta^{\bDelta} < n_\alpha^{\bDelta}$).

By using arguments analogous to those used in the proof of
Proposition~\ref{prop_no_branch_rel_A} (using the above modified
definition of $n_\rho^{\bDelta}$) we prove the following.

\begin{proposition}
Let $\bDelta \in \tilde{\calA}$. Then there exists a gentle quiver
$\bDelta'$ without branch relations such that $k \bDelta$ and $k
\bDelta'$ are derived equivalent.
\end{proposition}

In the next step of our proof we get rid of the branch arrows and
the branch triangles.

\begin{proposition} \label{prop_nobranchesAtilde}
Let $\bDelta \in \tilde{\calA}$. Then there exists a gentle quiver
$\bDelta'$ without branch arrows and branch triangles such that $k
\bDelta$ and $k \bDelta'$ are derived equivalent.
\end{proposition}

\begin{proof}
According to the previous lemma we may assume that there are no
branch relations in $\bDelta$.

For a strongly cycle vertex $x$ we define the number $m_x'$ in the
following way: $m_x' := 0$ if either $x$ is adjacent to a cycle
triangle or for each strongly cycle arrow $\alpha$ such that $t
\alpha = x$ there is no $\beta \in \Delta_1$ such that $(\alpha,
\beta) \in R$. Otherwise, we put $m_x' := m_{s \alpha}' + 1$,
where $\alpha$ is the strongly cycle arrow terminating at $x$
(this definition is correct, since $\alpha \not \in
\calC_{\bDelta}$). We define $m_x''$ dually.

Let $\calV$ be the set of the strongly cycle vertices $x$ which
are adjacent either to a branch arrow or to a branch triangle. For
$x \in \calV$ we put $m_x := m_x'$ if either there is a branch
triangle adjacent to $x$ or there is a branch arrow terminating at
$x$. Otherwise, we put $m_x := m_x''$. Finally, let $m_{\bDelta}
:= \min \{ m_x \mid x \in \calV \}$ and denote by $r_{\bDelta}$
the sum of the numbers of the branch arrows and the branch
triangles in $\bDelta$.

Assume that $r_{\bDelta} >0$ and fix $x \in \calV$ with $m_x =
m_{\bDelta}$. Observe, that lack of branch relations in $\bDelta$
implies that there may be at most one branch arrow adjacent to
$x$. Consequently, by symmetry we may assume that if there is a
branch arrow adjacent to $x$, then it terminates at $x$. Let
$\alpha$ and $\beta$ be the strongly cycle arrows adjacent to $x$.
Since there are no branch relations in $\bDelta$, then (up to
symmetry) $s \alpha = x = t \beta$ and $(\alpha, \beta) \in R$.
Moreover, $\alpha \in \calC_{\bDelta}$ if and only if $\beta \in
\calC_{\bDelta}$. Finally, if $\alpha, \beta \in \calC_{\bDelta}$,
then they belong to the same triangle. Put $y := s \beta$.

Let $\bDelta'$ be the quiver obtained from $\bDelta$ by applying
the reflection at $x$ (we can apply the reflection at $x$ since
there are two arrows terminating at $x$). If $m_x = 0$, then one
easily checks that $r_{\bDelta'} < r_{\bDelta}$. On the other
hand, if $m_x > 0$ and there is a branch triangle adjacent to $x$,
then $r_{\bDelta'} = r_{\bDelta}$ and $m_{\bDelta'} <
m_{\bDelta}$. Finally, if $m_x > 0$ and there is a branch arrow
adjacent to $x$, then $r_{\bDelta''} = r_{\bDelta}$ and
$m_{\bDelta''} < m_{\bDelta}$, where $\bDelta''$ is the quiver
obtained from $\bDelta'$ by applying the reflection at $y$ (we can
apply the reflection at $y$ to $\bDelta'$, since in $\Delta'$
there are no arrows starting at $y$). Consequently, the claim
follows by induction.
\end{proof}

Let $\bDelta \in \tilde{\calA}$ and assume there are neither
branch arrows nor branch triangles in $\bDelta$. We investigate
its structure more closely.

First observe that for each triangle $\calO$ there exists uniquely
determined $\gamma_{\calO} \in \calO$ such that there are no
branch arrows in $\Delta \setminus \{ \gamma_{\calO} \}$. Indeed,
one easily observes that $\Delta \setminus \calO$ is not connected
(this follows since by removing one arrow from each triangle we
get a 1-cycle quiver). Now $\gamma_{\calO}$ is the arrow in
$\calO$ such that $\Delta \setminus (\calO \setminus \{
\gamma_{\calO} \})$ is still not connected. It follows from the
definition of cycle triangles that $\gamma_{\calO}$ is uniquely
determined and has the desired property.

Let $\bGamma := \bDelta \setminus \{ \gamma_{\calO} \mid \calO \in
\calC_{\bDelta} / \bbZ \}$. We call $\bGamma$ the standard model
of $\bDelta$. Note that $\Gamma$ is a cycle. If $\alpha \in
\Delta_1 \setminus \calC_{\bDelta}$, then we say that $\alpha$ is
clockwise (anticlockwise) oriented if $\alpha$ is clockwise
(anticlockwise, respectively) oriented in $\Gamma$. Similarly,
$\calO \in \calC_{\bDelta} / \bbZ$ is said to be clockwise
(anticlockwise) oriented if $\calO \setminus \{ \gamma_{\calO} \}$
consists of clockwise (anticlockwise, respectively) oriented
arrows in $\Gamma$.

By a free relation in $\bDelta$ we mean a relation $(\alpha,
\beta) \in R$ such that $\alpha, \beta \not \in \calC_{\bDelta}$.
We claim that if there are free relations in $\bDelta$, then the
relations in $\bGamma$ cannot be equioriented. Indeed, if, for
example, the relations in $\bGamma$ are clockwise oriented, then
$f_{\bGamma} = [p - r, p] + [q + r, q]$ according to
Proposition~\ref{prop_interpretation}, where $p$ and $q$ are the
numbers of the clockwise and the anticlockwise oriented arrows in
$\bGamma$, respectively, and $r$ is the number of the relations in
$\bGamma$. Consequently, an application of
Corollary~\ref{coro_f_completion} implies that
\[
f_{\bDelta} = m \cdot [0, 3] + [p - r - m, p - 2 m] + [q + r, q],
\]
where $m$ is the number of the completed relations. Now $p - r - m
\geq p - 2 m$, since $\bDelta \in \tilde{\calA}$, hence $r = m$,
which means that there are no free relations in $\bDelta$, and
this finishes the proof of the claim.

In the next step of our proof we eliminate the free relations.

\begin{proposition} \label{prop_finalreduction Atilde}
Let $\bDelta \in \tilde{\calA}$. Then there exists a gentle quiver
$\bDelta'$ without branch arrows, branch triangles and free
relations derived equivalent to $\bDelta$.
\end{proposition}

\begin{proof}
According to Proposition~\ref{prop_nobranchesAtilde} we may assume
that there are neither branch arrows nor branch triangles in
$\bDelta$. We say that a free relation $(\alpha, \beta)$ in
$\bDelta$ is clockwise (anticlockwise) oriented if $\alpha$ and
$\beta$ are clockwise (anticlockwise, respectively) oriented. Let
$\bGamma$ be the standard model of $\bDelta$. For each free
relation $\rho$ in $\bDelta$ we denote by $k_\rho$ the minimal
distance between $\rho$ and a oppositely oriented relation in
$\bGamma$ (this it well-defined according to the above
considerations). We put
\[
k_{\bDelta} := \min \{ k_\rho \mid \text{$\rho$ is a free relation
in $\bDelta$} \}
\]
and denote by $s_{\bDelta}$ the number of the free relations in
$\bDelta$.

Assume that $s_{\bDelta} > 0$ and fix a free relation $(\alpha,
\beta)$ in $\bDelta$ with $k := k_{(\alpha, \beta)} =
k_{\bDelta}$. We may assume that $\rho$ is clockwise oriented.
Then we have the following subquiver of $\Gamma$, where $\bGamma$
is the standard model of $\bDelta$,
\[
\xymatrix{\bullet \ar[r]^\beta & \vertexU{y} \ar[r]^\alpha &
\vertexU{x_0} \ar@{-}[r]^{\gamma_1} &\vertexU{x_1} \ar@{-}[r] &
\cdots \ar@{-}[r] & \vertexU{x_{k - 1}} \ar@{-}[r]^{\gamma_k} &
\vertexU{x_k} & \vertexU{y'} \ar[l]_{\alpha'} & \bullet
\ar[l]_{\beta'}}
\]
such that $(\alpha', \beta') \in R$. The minimality of $k$ implies
that there is no $i \in [1, k - 1]$ such that $(\gamma_i,
\gamma_{i + 1}) \in R$. For the same reason, if $(\gamma_{i + 1},
\gamma_i) \in R$ for some $i \in [1, k - 1]$, then $\gamma_i,
\gamma_{i + 1} \in \calC_{\bDelta}$ and they belong to the same
triangle.

First we show that we may assume that the quiver is ordered in the
following sense: there exists $l \in [0, k]$ such that the
following conditions are satisfied:
\begin{enumerate}

\item
if $i \in [1, l]$, then $\gamma_i$ is clockwise oriented and
$\gamma_i \not \in \calC_{\bDelta}$,

\item
if $i \in [l + 1, k]$ and $\gamma_i$ is clockwise oriented, then
$\gamma_i \in \calC_{\bDelta}$.

\end{enumerate}
Indeed, assume not. Then there exists $i \in [1, k - 1]$ such that
$\gamma_{i + 1}$ is clockwise oriented and either $\gamma_i$ is
anticlockwise oriented or $\gamma_i \in \calC_{\bDelta}$. One
easily checks that by applying the reflection at $x_i$ we
``improve'' the configuration (i.e., we decrease the number of the
pairs $(i, j)$ such that $i, j \in [1, k]$, $i > j$, $\gamma_i$ is
clockwise oriented, $\gamma_i \not \in \calC_{\bDelta}$, and
either $\gamma_j$ is anticlockwise oriented or $\gamma_j \in
\calC_{\bDelta}$), hence the claim follows by induction.

Next, we show that we may assume that $0 = l = k$. Indeed, if $l >
0$, then by applying the reflection at $x_0$ we obtain the quiver
$\bDelta'$ (without branch arrows and branch triangles) with
$s_{\bDelta'} = s_{\bDelta}$ and $k_{\bDelta'} < k_{\bDelta}$.
Similarly, if $l < k$ and $\bDelta'$ is the quiver obtained from
$\bDelta$ by applying either the reflection at $x_k$ (if $\alpha'
\not \in \calC_{\bDelta}$) or the coreflection at $x_k$
(otherwise), then $s_{\bDelta'} = s_{\bDelta}$ and $k_{\bDelta'} <
k_{\bDelta}$ (note that if $l < k$ and $\gamma_k \in
\calC_{\bDelta}$, then the minimality of $k$ implies that $\alpha'
\in \calC_{\bDelta}$).

Now we have two cases to consider, depending on whether $\alpha'$
belongs or not to $\calC_{\bDelta}$. If $\alpha' \in
\calC_{\bDelta}$, then by applying the reflections at $x_0$ and
$y'$ we obtain the quiver $\bDelta'$ with $s_{\bDelta'} <
s_{\bDelta}$. On the other hand, if $\alpha' \not \in
\calC_{\bDelta}$, then we apply the reflections at $x_0$, $y$ and
$y'$, and we obtain the quiver $\bDelta'$ such that $s_{\bDelta'}
< s_{\bDelta}$. By induction this finishes the proof.
\end{proof}

In view of the above proposition the following claim finishes the
proof of Proposition~\ref{proposition_A_tilde}.

\begin{proposition}
Let $\bDelta \in \tilde{\calA}$ be a gentle quiver without branch
arrows, branch triangles, and free relations. Then $k \bDelta$ is
a cluster tilted algebra of type $\tilde{\bbA}$.
\end{proposition}

\begin{proof}
Let $\bGamma$ be the standard model of $\bDelta$. Lack of free
relations in $\bDelta$ implies that $\bDelta$ is obtained from
$\bGamma$ by completing all relations, and this finishes the proof
according to Proposition~\ref{proposition_cluster_A_tilde}.
\end{proof}

\section{Proof of Theorem~\ref{main_D}} \label{section_D}

First, we prove the only missing ingredient of the proof. If
$\Delta$ is a quiver, then by a $3$-cycle in $\Delta$ we mean
every sequence $(\alpha_0, \alpha_1, \alpha_2)$ of arrows such
that $\alpha_0 \neq \alpha_1 \neq \alpha_2 \neq \alpha_0$ and
$(\alpha_0, \alpha_1, \alpha_2, \alpha_0)$ is a path in $\Delta$.
We identify $3$-cycles which differ only by a cyclic permutation.

\begin{proposition} \label{proposition_BB_A}
Let $\bDelta \in \calA$ and $m$ be the number of the $3$-cycles in
$\Delta$. Then
\[
f_{\bDelta} = m \cdot [0, 3] + [|\Delta_0| + 1 - m, |\Delta_0| - 1
- 2 m].
\]
\end{proposition}

\begin{proof}
Let $\bDelta'$ be a model of $\bDelta$. Lemma~\ref{lemma_gentle_A}
implies that
\[
f_{\bDelta'} = [|\Delta_0'| + 1, |\Delta_0'| - 1].
\]
Since $|\Delta_0'| = |\Delta_0|$, we obtain using
Corollary~\ref{coro_f_completion} that
\[
f_{\bDelta} = m \cdot [0, 3] + [|\Delta_0| + 1 - m, |\Delta_0| - 1
- 2 m],
\]
where $m$ is the number of the triangles in $\bDelta$. Now, it is
easy to observe that there is a bijection between the triangles in
$\bDelta$ and the $3$-cycles in $\Delta$, which finishes the
proof. Indeed, if $\{ \alpha, \Psi_{\bDelta} \alpha,
\Psi_{\bDelta}^2 \alpha \}$ is a triangle in $\bDelta$, then
$(\alpha, \Psi_{\bDelta} \alpha, \Psi_{\bDelta}^2 \alpha)$ is a
$3$-cycle. On the other hand, assume that $(\alpha_0, \alpha_1,
\alpha_2)$ is a $3$-cycle in $\Delta$. If $\{ \alpha_0, \alpha_1,
\alpha_2 \}$ is not a triangle, then there exists a model
$\bDelta''$ of $\bDelta$ such that $\alpha_0, \alpha_1, \alpha_2
\in \Delta_1''$. However, $\bDelta''$ is not of tree type, which
contradicts Lemma~\ref{lemma_BV}.
\end{proof}

As an immediate consequence we obtain the following reformulation
of~\cite{BV}*{Theorem}.

\begin{corollary}
Let $\Lambda$ and $\Lambda'$ be cluster tilted algebras of type
$\bbA$. Then $\Lambda$ and $\Lambda'$ are derived equivalent if
and only if $f_\Lambda = f_{\Lambda'}$.
\end{corollary}

Now we can finish the proof of Theorem~\ref{main_D}. Let $\Lambda$
and $\Lambda'$ be gentle algebras derived equivalent to cluster
tilted algebras $C$ and $C'$, respectively.
Corollary~\ref{corollary_ABChJP} implies that $C$ and $C'$ are of
types $\bbA$ or $\tilde{\bbA}$. Now, using
Corollary~\ref{corollary_cluster_A} and
Proposition~\ref{proposition_A} (in the $\bbA$-case), and
Corollary~\ref{corollary_cluster_A_tilde} and
Proposition~\ref{proposition_A_tilde} (in the
$\tilde{\bbA}$-case), we get that we may assume that $\Lambda$ and
$\Lambda'$ are BB-equivalent to $C$ and $C'$, respectively.

(1)$\implies$(2) Assume that $\Lambda$ and $\Lambda'$ are derived
equivalent. Then $C$ and $C'$ are derived equivalent, hence $C$
and $C'$ are BB-equivalent according to
Theorems~\ref{theorem_BB_A} ($\bbA$-case)
and~\ref{theorem_BB_A_tilde} ($\tilde{\bbA}$-case), thus also
$\Lambda$ and $\Lambda'$ are BB-equivalent.

(2)$\implies$(3) Follows from Theorem~\ref{theorem_AAG}.

(3)$\implies$(1) Assume that $f_\Lambda = f_{\Lambda'}$. Then $f_C
= f_{C'}$, hence $C$ and $C'$ are derived equivalent according to
Proposition~\ref{proposition_BB_A} ($\bbA$-case) and
Theorem~\ref{theorem_BB_A_tilde} ($\tilde{\bbA}$-case).
Consequently, $\Lambda$ and $\Lambda'$ are also derived
equivalent, which finishes the proof.

\section{Proof of Theorem~\ref{main_E}} \label{section_E}

Recall that the gentle algebras are Gorenstein. More precisely, we
have the following~\cite{GR}*{Theorem 3.4}.

\begin{theorem} \label{theo_Geiss_Reiten}
Let $\Lambda$ be the path algebra of a gentle quiver $\Delta$.
Then \[ \Gdim \Lambda = \max \{ \ell (\omega) \mid \omega \in
\calN_{\bDelta} \}
\]
if $\calN_{\bDelta} \neq \varnothing$, and $\Gdim \Lambda \leq 1$,
otherwise. In particular, $\Gdim \Lambda \leq 1$ if and only if
$\ell (\omega) \leq 1$ for each $\omega \in \calN_{\bDelta}$.
\end{theorem}

Now we describe the algebras derived equivalent to cluster tilted
algebras of type $\bbA$ in terms of their quivers.

\begin{proposition}
An algebra $\Lambda$ is derived equivalent to a cluster tilted
algebra of type $\bbA$ if and only if there exists a gentle quiver
$\bDelta$ of tree type and a subset $R_0 \subset R$ consisting of
isolated relations, such that $\Lambda$ is isomorphic to the path
algebra of the quiver obtained from $\bDelta$ by completing the
relations from $R_0$.
\end{proposition}

\begin{proof}
If $\Lambda$ is derived equivalent to a cluster tilted algebra of
type $\bbA$, then $\Lambda$ is of the form described in the
proposition due to Lemma~\ref{lemma_BV}. On the other hand, if
$\bDelta$ is a gentle quiver of tree type, $R_0 \subset R$
consists of isolated relations, and $\bDelta'$ is obtained from
$\bDelta$ by completing the relations from $R_0$, then
\[
f_{\bDelta'} = |R_0| \cdot [0, 3] + [|\Delta_0| + 1 - |R_0|,
|\Delta_0| - 1 - 2 |R_0|]
\]
according to Lemma~\ref{lemma_gentle_A} and
Corollary~\ref{coro_f_completion}.
\end{proof}

We may characterize cluster tilted algebras of type $\bbA$ among
the above class of algebras in the following way.

\begin{corollary} \label{coro_Gdim_A}
Let $\bDelta$ be a gentle quiver of tree type, $R_0 \subset R$ a
subset consisting of isolated relations, and $\Lambda$ the path
algebra of the quiver obtained from $\bDelta$ by completing the
relations from $R_0$. Then the following conditions are
equivalent:
\begin{enumerate}

\item
$\Lambda$ is cluster tilted.

\item
$\Gdim \Lambda \leq 1$.

\item
$R_0 = R$.

\end{enumerate}
\end{corollary}

\begin{proof}
(1)$\implies$(2) Follows from Theorem~\ref{theo_Keller_Reiten}.

(2)$\implies$(3) Let $\bDelta'$ be the quiver obtained from
$\bDelta$ by completing the relations from $R_0$. According to
Theorem~\ref{theo_Geiss_Reiten}, $\Gdim \Lambda \leq 1$ implies
that $\ell (\omega) \leq 1$ for each $\omega \in
\calN_{\bDelta'}$, hence $R_0 = R$ according to
Lemma~\ref{lemma_completion} (note that $\calC_{\bDelta} =
\varnothing$ since $\bDelta$ is of tree type).

(3)$\implies$(1) Follows from
Proposition~\ref{proposition_cluster_A}.
\end{proof}

Now we study $\tilde{\bbA}$ case. We start with the following.

\begin{proposition} \label{prop_9.5}
Let $\Lambda$ be an algebra. If $\Lambda$ is derived equivalent to
a cluster tilted algebra of type $\tilde{\bbA}$, then there exists
a 1-cycle gentle quiver $\bDelta$ and a subset $R_0 \subset R$
consisting of isolated relations, such that $\Lambda$ is
isomorphic to the path algebra of the quiver obtained from
$\bDelta$ by completing the relations from $R_0$.
\end{proposition}

\begin{proof}
Follows immediately from Lemma~\ref{lemma_erase_A_tilde}.
\end{proof}

The converse implication is not true in general. Namely, if
$\bDelta$ is the following quiver with relations
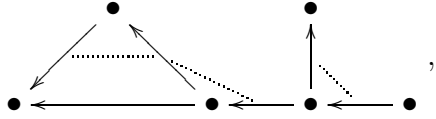
\[
\vcenter{\xymatrix{& \bullet \ar[ld]^{}="a1"  & & \bullet \\
\bullet & & \bullet \ar[lu]^{}="a2"_{}="b1" \ar[ll] & \bullet
\ar[l]_{}="b2" \ar[u]_{}="c1" & \bullet \ar[l]_{}="c2"
\ar@{.}"a2";"a1" \ar@{.}"b2";"b1" \ar@{.}"c2";"c1"}},
\]
then $\bDelta$ is a 1-cycle gentle quiver (in particular, it is of
the above form for $R_0 = \varnothing$), but $f_{\bDelta} = [4, 5]
+ [2, 1]$, so $\bDelta$ is not derived equivalent to a cluster
tilted algebra. Note however that, if $\bDelta'$ is the following
quiver with relations
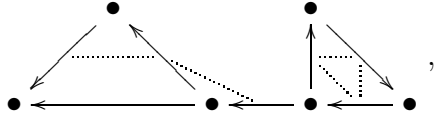
\[
\vcenter{\xymatrix{& \bullet \ar[ld]^{}="a1"  & & \bullet
\ar[rd]_{}="c3" \\ \bullet & & \bullet \ar[lu]^{}="a2"_{}="b1"
\ar[ll] & \bullet \ar[l]_{}="b2" \ar[u]_{}="c1" & \bullet
\ar[l]_{}="c2" \ar@{.}"a2";"a1" \ar@{.}"b2";"b1" \ar@{.}"c2";"c1"
\ar@{.}"c3";"c2" \ar@{.}"c1";"c3"}},
\]
then $f_{\bDelta'} = [0, 3] + [3, 3] + [2, 1]$, hence $\bDelta'$
is derived equivalent to a cluster tilted algebra. In general, in
order to obtain a converse of Proposition~\ref{prop_9.5}, one
would need to make assumptions on the relations: The number of the
completed clockwise (anticlockwise) relations, must be bigger than
the number of the anticlockwise (clockwise, respectively) cycle
relations. For this to make sense, one would need an appropriate
definition of orientation of branch relations.

Next, we obtain the following analogue of
Corollary~\ref{coro_Gdim_A} in the $\tilde{\bbA}$-case.

\begin{corollary} \label{coro_Gdim_Atilde}
Let $\bDelta$ be a 1-cycle gentle quiver, $R_0 \subset R$ a subset
consisting of isolated relations, and $\Lambda$ the path algebra
of the quiver obtained from $\bDelta$ by completing the relations
from $R_0$. If $\Lambda$ is derived equivalent to a cluster tilted
algebra of type $\tilde{\bbA}$, then the following conditions are
equivalent:
\begin{enumerate}

\item
$\Lambda$ is cluster tilted.

\item
$\Gdim \Lambda \leq 1$.

\item
$R_0 = R$.

\end{enumerate}
\end{corollary}

\begin{proof}
(1)$\implies$(2) Follows from Theorem~\ref{theo_Keller_Reiten}.

(2)$\implies$(3) First we show that $\calC_{\bDelta} =
\varnothing$. Indeed, if $\calC_{\bDelta} \neq \varnothing$, then
$f_{\bDelta} = [0, q] + [p + q, p]$ for some $p \in \bbN$ and $q
\in \bbN_+$ according to Proposition~\ref{prop_1_cycle}.
Consequently,
\[
f_\Lambda = m \cdot [0, 3] + [0, q] + [p + q - m, p - 2 m],
\]
where $m := |R_0|$, by Corollary~\ref{coro_f_completion}. Thus
Theorem~\ref{main_B} implies that $\Lambda'$ is not equivalent to
a cluster tilted algebra of type $\tilde{\bbA}$, which contradicts
our assumptions. This finishes the proof of our claim. Now, $\Gdim
\Lambda \leq 1$ implies that $\ell (\omega) \leq 1$ for each
$\omega \in \calN_{\bDelta'}$ according to
Theorem~\ref{theo_Geiss_Reiten}, hence $R_0 = R$ according to
Lemma~\ref{lemma_completion}.

(3)$\implies$(1) Follows from
Proposition~\ref{proposition_cluster_A_tilde}.
\end{proof}

Finally, we note that Theorem~\ref{main_E} follows immediately
from Corollaries~\ref{coro_Gdim_A} and~\ref{coro_Gdim_Atilde}.

\bibsection

\end{document}